\documentclass[a4paper, 10pt]{amsart}
\usepackage{amsmath}
\usepackage{amssymb}
\usepackage{amsthm}
\usepackage{enumerate}

\DeclareMathOperator{\Img}{Im}
\DeclareMathOperator{\Ker}{Ker}

\providecommand{\abs}[1]{\lvert#1\rvert}

\newtheorem{theorem}{Theorem}
\newtheorem{prop}{Proposition}
\newtheorem{lemma}{Lemma}
\newtheorem{rem}{Remark}

\begin{document}

\title[Automorphisms of graphs of finite-rank self-adjoint operators]{Automorphisms of graphs  corresponding to conjugacy classes of finite-rank self-adjoint operators}
\author{Mark Pankov, Krzysztof Petelczyc, Mariusz \.Zynel}

\address{Mark Pankov: Faculty of Mathematics and Computer Science, 
University of Warmia and Mazury, S{\l}oneczna 54, 10-710 Olsztyn, Poland}
\email{pankov@matman.uwm.edu.pl}
\address{Krzysztof Petelczyc, Mariusz \.Zynel: Faculty of Mathematics, University of Bia{\l}ystok, Cio{\l}kowskiego 1 M, 15-245 Bia{\l}ystok, Poland}
\email{kryzpet@math.uwb.edu.pl, mariusz@math.uwb.edu.pl}

\maketitle

\begin{abstract}
We consider the graph whose vertex set is a conjugacy class ${\mathcal C}$ consisting of  finite-rank self-adjoint operators on a complex Hilbert space $H$.
The dimension of $H$ is assumed to be not less than $3$.
In the case when operators from ${\mathcal C}$ have two eigenvalues only, we obtain the Grassmann graph formed by $k$-dimensional subspaces of $H$,
where $k$ is the smallest dimension of 
eigenspaces.
Classical Chow's theorem describes automorphisms of this graph for $k>1$.
Under the assumption that operators from ${\mathcal C}$ have more than two eigenvalues
we show that every automorphism of the graph is induced by a unitary or anti-unitary operator
up to a permutation of eigenspaces with the same dimensions.
In contrast to this result, Chow's theorem states that there are graph automorphisms induced
by semilinear automorphisms not preserving orthogonality if ${\mathcal C}$ is formed by operators with precisely two eigenvalues.
\end{abstract}

\section{Introduction}
The {\it Grassmann graph} is the simple graph whose vertices are $m$-dimensional subspaces of a vector space $V$ (over a field) and 
two such subspaces are adjacent (connected by an edge) if their intersection is $(m-1)$-dimensional.
Chow's theorem \cite{Chow} describes automorphisms of this graph under the assumption that $1<m<\dim V-1$ 
(in the two remaining cases
the Grassmann graph is a complete graph, and thus
every bijective transformation of the vertex set is a graph automorphism).
The theorem states that every automorphism is induced by a semilinear automorphism of $V$
or, only when $\dim V=2m$, by a semilinear isomorphism of $V$ to the dual vector space $V^*$. 
The original version of Chow's theorem concerns Grassmann graphs of finite-dimensional vector spaces,
but the same holds for the infinite-dimensional case (see, for example, \cite[Chapter 2]{Pankov-book2}).

The Grassmannian of $m$-dimensional subspaces of a complex Hilbert space $H$
can be naturally identified with the conjugacy class of rank-$m$ projections.
Note that projections can be characterised as self-adjoint idempotents in the algebra of bounded operators
and play an important role in operator theory and mathematical foundations of quantum mechanics.
By Gleason's theorem \cite{Gleason}, rank-one projections correspond to pure states of quantum mechanical systems.
Classical Wigner's theorem \cite{Wigner} describes symmetries of the space of pure states, 
i.e.\ the conjugacy class of rank-one projections.
Moln\'ar \cite{M1} extended this result  on other conjugacy classes of finite-rank projections.

Two $m$-dimensional subspaces of $H$ are adjacent if and only if the difference of the corresponding projections is an operator of rank $2$, 
i.e.\ of the smallest possible rank (the rank of the difference of two finite-rank self-adjoint operators from the same conjugacy class cannot be equal to $1$).
Chow's theorem reformulated in terms of  projections was successfully 
applied 
to prove some Wigner-type theorems \cite{GeherSemrl, Geher,Pankov1,PV}
(see \cite{Pankov-book2} for a detailed description of the topic).

In \cite{PPZ}, 
the authors extend the 
adjacency relation, described above, 
from conjugacy classes of projections to 
conjugacy classes of finite-rank self-adjoint operators.
By the spectral theorem (its finite-dimensional version) every
finite-rank self-adjoint operator on $H$ is a real linear combination of finite-rank projections whose images are mutually orthogonal
(these images are the eigenspaces of the operator). 
A conjugacy class of such operators is completely determined by 
the spectrum and the dimensions of the eigenspaces.
If the spectrum consists of $k$ eigenvalues, then
the associated graph has $\binom{k}{2}$ different types of adjacency corresponding to pairs of eigenvalues;
in particular, we obtain a Grassmann graph if there are two eigenvalues only.
Let us have a look at the case of more than two eigenvalues.
Some combinatorial properties of the graph (for example, connectedness and cliques) are investigated in \cite{PPZ},
the main result concerns graph automorphisms under the assumption that the dimensions of all eigenspaces are greater than $1$
(in Section 4 we explain why reasonings from \cite{PPZ} do not work for the general case).
In the present paper, we determine graph automorphisms without any assumption.
Using a new approach based on Johnson graphs we show that 
every automorphism is induced by a unitary or anti-unitary operator up to a permutation of eigenspaces with the same dimensions.
For Grassmann graphs (the case of two eigenvalues) the above statement 
fails, as there are graph automorphisms induced
by semilinear automorphisms of $H$ that do not preserve orthogonality.

\section{Result} 

Let $H$ be a complex Hilbert space of dimension not less than $3$. 
The projection on a closed subspace $X\subset H$ will be denoted by $P_X$. 
Two operators $A$ and $B$ on $H$ are (unitary) conjugate if there is a unitary operator $U$ on $H$ such that $B=UAU^*$.
A conjugacy class of finite-rank self-adjoint operators on $H$ is completely determined by 
the spectrum $\sigma=\{a_i\}_{i\in I}$ of operators from this class  and the family $d=\{n_i\}_{i\in I}$,
where $n_i$ 
is
the dimension of eigenspaces corresponding to the eigenvalue $a_i$. 
This conjugacy class will be denoted by ${\mathcal G}(\sigma,d)$. 
Note that the index set $I$ is finite, each $a_i$ is real and at most one of $n_i$ is infinite.
If $n_i$ is infinite, then $a_i=0$, i.e.\ the corresponding eigenspaces are the kernels of operators from ${\mathcal G}(\sigma,d)$.
For every $A\in {\mathcal G}(\sigma,d)$ we have
$$A=\sum_{i\in I}a_{i}P_{X_i},$$
where $\{X_i\}_{i\in I}$ is a collection of mutually orthogonal subspaces whose sum coincides with $H$
and $\dim X_i=n_i$ for all $i\in I$. 

Operators $A,B\in {\mathcal G}(\sigma, d)$
are said to be {\it adjacent} if the following conditions are satisfied:
\begin{enumerate}[({A}1)]
\item the rank of $B-A$ is equal to $2$;
\item $\Img(B-A)$ and $\Ker(B-A)$ are invariant to both $A$ and $B$.
\end{enumerate}
The condition (A1) implies (A2) if $\abs{I}=2$ \cite[Example 2]{PPZ}.
If $\abs{I}\ge 3$, then there are pairs $A,B\in {\mathcal G}(\sigma, d)$ satisfying (A1) and such that (A2) fails \cite[Examples 4, 5]{PPZ}.

There is a geometric interpretation of operator adjacency.
Let $X,Y\subset H$ be closed subspaces of the same finite dimension or the same finite codimension.
We say that $X$ and $Y$ are {\it adjacent} if $X\cap Y$ is a hyperplane in both $X,Y$ or, equivalently, $X,Y$ both are hyperplanes in $X+Y$.
Observe that $X,Y$ are adjacent if and only if $X^{\perp},Y^{\perp}$ are adjacent.
In the case when $I=\{1,2\}$,
operators $A,B\in {\mathcal G}(\sigma, d)$ are adjacent ($A-B$ is of rank $2$) if and only if
the eigenspaces of $A,B$ corresponding to a certain $a_i$ are adjacent; 
the latter immediately implies  that the eigenspaces corresponding to $a_{3-i}$ also are adjacent.
Consider the general case.
Let $X_i$ and $Y_i$ be the eigenspaces of $A$ and $B$ (respectively) corresponding to $a_i$. 
The conditions (A1) and (A2) hold if and only if there are distinct $i,j\in I$ such that the following assertions are fulfilled:
\begin{itemize}
\item $X_i$ and $X_j$ are adjacent to $Y_i$ and $Y_j$, respectively;
\item $X_t=Y_t$ for all $t\in I\setminus \{i,j\}$.
\end{itemize}
In this case, we say that the operators $A,B$ are $(i,j)$-{\it adjacent}.

Let $\Gamma(\sigma,d)$ be the simple graph whose vertex set is  ${\mathcal G}(\sigma, d)$
and two operators are connected by an edge in this graph if they are adjacent. 
This graph is connected \cite[Theorem 2]{PPZ}.
If $\sigma'=\{a'_i\}_{i\in I}$ is a collection of mutually distinct real numbers such that $a'_i=0$ if $n_i$ is infinite,
then the map
$$\sum_{i\in I}a_iP_{X_i}\to \sum_{i\in I}a'_iP_{X_i}$$
is an isomorphism between $\Gamma(\sigma,d)$ and $\Gamma(\sigma',d)$.

Note that when $I=\{1,2\}$ and $n_1$ is finite, 
$\Gamma(\sigma,d)$ is isomorphic to the Grassmann graph of $n_1$-dimensional subspaces of $H$.
If $n_2$ is finite, then $\Gamma(\sigma,d)$ also is isomorphic to the Grassmann graph formed by $n_2$-dimensional subspaces\footnote{
If $H$ 
is finite-dimensional, the orthocomplementation map provides an isomorphism between the Grassmann graphs formed 
by $m$-dimensional and $(\dim H-m)$-dimensional subspaces of $H$.}.

Let $U$ be a unitary or anti-unitary operator on $H$. 
Consider the bijective transformation sending every $A\in {\mathcal G}(\sigma, d)$
to $UAU^{*}$. If $A=\sum_{i\in I}a_{i}P_{X_i}$, then
$$UAU^*=\sum_{i\in I}a_{i}P_{U(X_i)}$$
which shows that this transformation is an automorphism of the graph $\Gamma(\sigma,d)$.

Let $S(d)$ be the group of all 
permutations 
$\delta$ on $I$ satisfying $n_i=n_{\delta(i)}$.
For every $\delta\in S(d)$ and $A=\sum_{i\in I}a_{i}P_{X_i}\in {\mathcal G}(\sigma, d)$
the operator 
$$\delta(A)=\sum_{i\in I}a_{i}P_{X_\delta(i)}$$
also belongs to ${\mathcal G}(\sigma, d)$. 
The transformation sending every $A\in {\mathcal G}(\sigma, d)$
to $\delta(A)$ is an automorphism of $\Gamma(\sigma,d)$.

\begin{theorem}\label{theorem-main}
If\/ $\abs{I}\ge 3$ and
$f$ is an automorphism of\/ 
$\Gamma(\sigma, d)$, then
there are a unitary or anti-unitary operator $U$ on $H$ and a permutation  $\delta\in S(d)$
such that 
$$f(A)=U\delta(A)U^{*}$$
for all $A\in {\mathcal G}(\sigma, d)$.
\end{theorem}

\begin{rem}{\rm
Suppose that $I=\{1,2\}$ and $n_1\le n_2$.
Then $\Gamma(\sigma,d)$ is isomorphic to the Grassmann graph of $n_1$-dimensional subspaces. 
Let $S$ be a semilinear automorphism of $H$, i.e.\ 
$$S(x+y)=S(x)+S(y)$$
for all $x,y\in H$ and there is an automorphism $\alpha$ of the field ${\mathbb C}$ such that
$$S(ax)=\alpha(a)S(x)$$
for all $a\in {\mathbb C}$ and $x\in H$.
The bijective transformation sending every $a_{1}P_{X_1}+a_2P_{X_2}$ from ${\mathcal G}(\sigma,d)$ to
$$a_{1}P_{S(X_1)}+a_2P_{S(X_1)^{\perp}}$$
is a graph automorphism. 
If $n_1=n_2$, then the same holds for the transformation sending every $a_{1}P_{X_1}+a_2P_{X_2}\in {\mathcal G}(\sigma,d)$ to
$$a_{1}P_{S(X_1)^{\perp}}+a_2P_{S(X_1)}$$
(the transposition of the eigenspaces). 
By Chow's theorem \cite{Chow}, there are no other graph automorphisms if $n_1>1$.
In the case when $n_1=1$, any two operators from the conjugacy class are adjacent and every bijective transformation is a graph automorphism.
}\end{rem}

\begin{rem}\upshape
In \cite{PPZ}, Theorem \ref{theorem-main} is proved under 
the additional assumption that $n_i>1$ for all $i\in I$.
It is required due to applied arguments. 
In Section 4, we explain why these arguments fail in the general case.
\end{rem}

\section{Proof of Theorem \ref{theorem-main}}
Throughout this section we assume that $\abs{I}\ge 3$. 

\subsection{One technical result}
For every integer $m$ satisfying $0<m<\dim H$ we denote by ${\mathcal G}_m(H)$ the Grassmannian formed by $m$-dimensional subspaces of $H$. 

\begin{prop}\label{prop-Chow}
Let $m$ and $l$ be integer satisfying $0<m<l<\dim H$. If $g$ and $h$ are bijective transformations of ${\mathcal G}_m(H)$ and ${\mathcal G}_l(H)$
{\rm(}respectively{\rm)} such that for any $X\in {\mathcal G}_m(H)$ and $Y\in {\mathcal G}_l(H)$ we have
$$X\subset Y\;\Longleftrightarrow\; g(X)\subset h(Y),$$
then $g$ and $h$ are induced by the same semilinear automorphism of $H$, i.e.\
there is a semilinear automorphism $S$ of $H$ such that
$$g(X)=S(X)\qquad\text{and}\qquad h(Y)=S(Y)$$
for all $X\in {\mathcal G}_m(H)$ and $Y\in {\mathcal G}_l(H)$.
\end{prop}

\begin{proof}
In  \cite[Proposition 3.4]{Pankov-book1}, the statement is proved for finite-dimensional vector spaces, 
but the same reasonings work in the general case.
\end{proof}

\subsection{$(i,j)$-connected components}
Let $i$ and $j$ be distinct indices from $I$.
We say that operators $A,B\in {\mathcal G}(\sigma,d)$ are $(i,j)$-{\it connected} if there is a sequence of operators
$$
A=C_{0},C_{1},\dots, C_{m}=B,
$$
where $C_{t-1},C_{t}$ are $(i,j)$-adjacent for all $t\in \{1,\dots,m\}$.
Two operators from ${\mathcal G}(\sigma,d)$ are $(i,j)$-connected if and only if 
for every $p\in I\setminus\{i,j\}$ they have the same eigenspace corresponding to $a_p$
\cite[Lemma 1]{PPZ}.
An {\it $(i,j)$-connected component} of $\Gamma(\sigma,d)$ is a subset of ${\mathcal G}(\sigma,d)$
maximal with respect to the property that any two operators are $(i,j)$-connected.

At least one of $n_i,n_j$ is finite. 
Suppose that $n_i$ is finite  and consider the pair $(\sigma,d)_{-i,+j}$ obtained from $(\sigma,d)$ as follows:
$n_j$ is replaced  by $n_j+n_i$ and $a_i,n_i$ are removed respectively from $\sigma$ and $d$.
This pair defines the conjugacy class of finite-rank self-adjoint operators ${\mathcal G}\bigl((\sigma,d)_{-i,+j}\bigr)$
whose spectrum is $\sigma\setminus \{a_i\}$, the dimension of eigenspaces corresponding to $a_{j}$ is $n_i+n_j$
and the dimensions of eigenspaces corresponding to the remaining $a_t$ are $n_t$.

\begin{rem}{\rm
If $n_i$ is infinite, then $a_i=0$ and, consequently, $a_j\ne 0$. 
Since eigenspaces corresponding to non-zero eigenvalues of finite-rank operators cannot be infinite-dimensional, 
we do not obtain a conjugacy class of finite-rank self-adjoint operators in this case.
}\end{rem}

For every  operator $T\in {\mathcal G}\bigl((\sigma,d)_{-i,+j}\bigr)$ we denote by ${\mathcal G}(T)$
the set of all operators $A\in{\mathcal G}(\sigma, d)$ such that 
$$A=T+(a_i -a_j) P_{X},$$
where $X$ is an $n_i$-dimensional subspace in the eigenspace of $T$ 
corresponding to $a_j$.
In other words, ${\mathcal G}(T)$ stands for 
all $A\in {\mathcal G}(\sigma, d)$ satisfying the following conditions:
\begin{itemize}
\item for every $t\in I\setminus\{i,j\}$ the eigenspaces of $A$ and $T$ corresponding to $a_t$  are coincident;
\item the eigenspace of $T$ corresponding to $a_j$ is  the orthogonal sum of the eigenspaces of $A$ corresponding to $a_i$ and $a_j$.
\end{itemize}
Note that for every $A\in {\mathcal G}(\sigma, d)$ there is a unique $T\in {\mathcal G}\bigl((\sigma,d)_{-i,+j}\bigr)$ 
such that $A\in {\mathcal G}(T)$. 

Now, assume that $n_j$ is also finite and consider the operator $Q\in {\mathcal G}\bigl((\sigma,d)_{-j,+i}\bigr)$ such that 
for every $t\in I\setminus\{i,j\}$ the eigenspaces of $Q$ and $T$ corresponding to $a_t$ are coincident 
and the eigenspace of $Q$ corresponding to $a_i$ coincides with the eigenspace of $T$ corresponding to $a_j$.
Then ${\mathcal G}(T)={\mathcal G}(Q)$.

By \cite[Lemma 2]{PPZ}, 
$$\bigl\{{\mathcal G}(T): T\in {\mathcal G}\bigl((\sigma,d)_{-i,+j}\bigr)\bigr\}$$
is the family of all $(i,j)$-connected components of ${\mathcal G}(\sigma,d)$.

\begin{rem}{\rm
The restriction of the graph $\Gamma(\sigma,d)$ to every $(i,j)$-connected component 
is isomorphic to the Grassmann graph formed by $n_i$-dimensional subspaces of $H'$, 
where $H'$ is a complex Hilbert space of dimension $n_i+n_j$.
}\end{rem}

\begin{lemma}\label{lemma0}
Operators $T,Q\in {\mathcal G}\bigl((\sigma,d)_{-i,+j}\bigr)$ are adjacent if and only if there are adjacent operators $A\in{\mathcal G}(T)$ and $B\in{\mathcal G}(Q)$.
\end{lemma}

\begin{proof}
The first part of Lemma 4 in \cite{PPZ}.
\end{proof}

\begin{rem}\label{rem-ad}{\rm
If $T,Q\in {\mathcal G}\bigl((\sigma,d)_{-i,+j}\bigr)$ are $(t,s)$-adjacent for some $t,s\in I\setminus\{i,j\}$,
then any adjacent $A\in{\mathcal G}(T)$ and $B\in{\mathcal G}(Q)$ also are  $(t,s)$-adjacent.
In the case when $T,Q\in {\mathcal G}\bigl((\sigma,d)_{-i,+j}\bigr)$ are $(t,j)$-adjacent for a certain $t\in I\setminus\{i,j\}$,
there are $(t,j)$-adjacent $A\in{\mathcal G}(T)$ and $B\in{\mathcal G}(Q)$ as well as $(t,i)$-adjacent $A'\in{\mathcal G}(T)$ and $B'\in{\mathcal G}(Q)$.
}\end{rem}

\subsection{Relations to automorphisms of the Johnson graph}

Let $f$ be an automorphism of the graph $\Gamma(\sigma,d)$.

\begin{lemma}\label{lemma1}
For any distinct $i,j\in I$ there are distinct $i',j'\in I$
such that $f$ sends every $(i,j)$-connected component to an $(i',j')$-connected component 
and $f^{-1}$ sends every $(i',j')$-connected component to an $(i,j)$-connected component.
Furthermore, we have $n_{i'},n_{j'}>1$ if and only if $n_i,n_j>1$.
\end{lemma}

\begin{proof}
Lemma 9 in \cite{PPZ}.
\end{proof}

By Lemma \ref{lemma1}, $f$ induces a permutation $\tau$ on the set of all $2$-element subsets of the index set $I$. 
Recall that the {\it Johnson graph} $J(I,2)$ is the simple graph whose vertex set is formed by all $2$-element subsets of $I$
and two distinct subsets are adjacent vertices in this graph if their intersection is non-empty.

\begin{lemma}\label{lemma2}
The transformation $\tau$ is an automorphism of $J(I,2)$.
\end{lemma}

The proof is based on the following technical lemma. 

\begin{lemma}\label{lemma2-tech}
The following assertions are fulfilled:
\begin{enumerate}[{\upshape (1)}]
\item Suppose that $\abs{I}\ge 4$ and for some mutually distinct $i,j,i',j'\in I$ there are operators $A,B,C\in {\mathcal G}(\sigma,d)$ 
such that $C$ is $(i,j)$-adjacent to $A$ and $(i',j')$-adjacent to $B$.
Then  there is an operator $C'\in {\mathcal G}(\sigma,d)$ which is $(i',j')$-adjacent to $A$ and $(i,j)$-adjacent to $B$.
\item For any mutually distinct $i,j,t\in I$ there are operators $A,B,C\in {\mathcal G}(\sigma,d)$
such that $C$ is $(i,j)$-adjacent to $A$ and $(j,t)$-adjacent to $B$ and there is no operator from ${\mathcal G}(\sigma,d)$
which is $(j,t)$-adjacent to $A$ and $(i,j)$-adjacent to $B$.
\end{enumerate}
\end{lemma}

\begin{proof}
(1). For every $s\in I\setminus\{i,j,i',j'\}$ the eigenspaces of $A,B,C$ corresponding to $a_s$ are coincident.
For $s\in \{i,j,i',j'\}$ we denote by $X_s$ and $Y_s$ the eigenspaces of $A$ and $B$, respectively.  
The eigenspaces of $C$ corresponding to $a_s$, $s\in\{i,j\}$ and $s\in \{i',j'\}$ coincide with $Y_s$ and $X_s$ (respectively).
This means that $X_s, Y_s$ are adjacent for all $s\in \{i,j,i',j'\}$;
moreover
$$X_i+X_j=Y_i+Y_j\quad\text{and}\quad X_{i'}+X_{j'}=Y_{i'}+Y_{j'}.$$
Consider the operators $C'\in{\mathcal G}(\sigma,d)$ whose eigenspace corresponding to $a_{s}$, $s\in I\setminus\{i,j,i',j'\}$  coincides with $X_s=Y_s$
and the eigenspaces corresponding to $a_s$, $s\in\{i,j\}$ and $s\in \{i',j'\}$ coincide with $X_s$ and $Y_s$ (respectively).
This operator is as required.

(2). Suppose that $A,B,C$ are operators from ${\mathcal G}(\sigma,d)$ such that $C$ is $(i,j)$-adjacent to $A$ and $(j,t)$-adjacent to $B$.
If $s\in I\setminus\{i,j,t\}$, then the eigenspaces of $A,B,C$ corresponding to $a_s$ are coincident.
For $s\in \{i,j,t\}$ we denote by $X_s$ and $Y_s$ the eigenspaces of $A$ and $B$, respectively. 
The eigenspaces of $C$ corresponding to $a_i$ and $a_t$ coincide with $Y_i$ and $X_t$ (respectively) and, consequently, 
these subspaces are orthogonal.  
If $C'\in{\mathcal G}(\sigma,d)$ is $(j,t)$-adjacent to $A$ and $(i,j)$-adjacent to $B$,
then the eigenspaces of $C'$ corresponding to $a_{i}$ and $a_t$ are $X_i$ and $Y_t$ (respectively) which implies that
$X_i,Y_t$ are orthogonal.

Now, for any operator $A\in {\mathcal G}(\sigma,d)$ whose eigenspace corresponding to $a_s$, $s\in I$
is denoted by $X_s$ we construct $B,C\in {\mathcal G}(\sigma,d)$ satisfying the required conditions.

Let us take any $1$-dimensional subspace $P\subset X_j$. 
Notice that it is the orthogonal complement of $X_i$ in $X_i+P$.
Next, take a  hyperplane $Y_i\subset X_i+P$ distinct from $X_i$ and denote 
by $Q$  the orthogonal complement of $Y_i$  in $X_i+P$ (it is $1$-dimensional).
It is clear that $X_i,Y_i$ are adjacent and $P\ne Q$, i.e.\ $Q$ is not orthogonal to $X_i$.
The subspace $Y_i$ is orthogonal to $X_t$ (since $X_i$ and $P\subset X_j$ both are orthogonal to $X_t$ and $Y_i\subset X_i+P$). 

Let $Y_t$ be an $n_t$-dimensional subspace containing $Q$ and adjacent to $X_t$.
Note that $Q$ is not contained in $X_t$ 
($Q$ is not orthogonal to $X_i$)
which implies that $Y_t=(X_t\cap Y_t)+Q$.
Thus, $Y_t$ is orthogonal to $Y_i$ (since $X_t$ and $Q$ both are orthogonal to $Y_i$).
We have
$$Y_i+Y_t \subset X_i+X_j+X_t$$
and write $Y_j$ for the orthogonal complement of $Y_i+Y_t$ in $X_i+X_j+X_t$.
Denote by $B$ the operator from ${\mathcal G}(\sigma,d)$ whose eigenspaces corresponding to $a_s$, $s\in I\setminus\{i,j,t\}$ and $s\in \{i,j,t\}$
coincide with $X_s$ and $Y_s$, respectively. 

We have $Y_i\subset X_i+X_j$, so
denote by $Z_j$ the orthogonal complement of $Y_i$ in $X_i+X_j$.
The subspaces $X_i,Y_i$ are adjacent and, consequently, $X_j,Z_j$ are adjacent  
(as the orthogonal complements of $X_i$ and $Y_i$ in $X_i+X_j$).
The inclusion $Z_j\subset X_i+X_j$ implies that $Z_j$ is orthogonal to $X_t$.
Consider the operator $C\in {\mathcal G}(\sigma,d)$ defined as follows:
\begin{itemize}
\item the eigenspace of $C$ corresponding to $a_s$, $s\in I\setminus\{i,j,t\}$  coincides with $X_s$;
\item the eigenspaces of $C$ corresponding to $a_i,a_j,a_t$ coincide with $Y_i,Z_j,X_t$ (respectively).
\end{itemize}
This operator is $(i,j)$-adjacent to $A$. Note that 
\begin{equation}\label{eq-1}
Z_j+X_t= Y_j+Y_t.
\end{equation}
 Since $X_t,Y_t$ are adjacent, $Z_j,Y_j$ are adjacent as the orthogonal complements of $X_t$ and $Y_t$ in \eqref{eq-1}.
Therefore, $C$ is $(j,t)$-adjacent to $B$.

Recall that $Q\subset Y_t$ is not orthogonal to $X_i$, i.e.\ $X_i$ and $Y_t$ are not orthogonal.
This means that there is no operator in ${\mathcal G}(\sigma,d)$ which is $(j,t)$-adjacent to $A$ and $(i,j)$-adjacent to $B$. 
\end{proof}

\begin{proof}[Proof of Lemma \ref{lemma2}]
If $\abs{I}=3$, then the intersection of any two distinct $2$-element subsets of $I$ is non-empty and the statement is trivial.

Let $\abs{I}\ge 4$ and  let $i,j,t\in I$ be mutually distinct indices. 
Suppose that 
$$\tau\bigl(\{i,j\}\bigr)=\{i',j'\}\qquad\text{and}\qquad\tau\bigl(\{j,t\}\bigr)=\{s',t'\}.$$
By the statement (2) from Lemma \ref{lemma2-tech}, 
there are $A,B,C\in {\mathcal G}(\sigma,d)$ such that $C$ is $(i,j)$-adjacent to $A$ and $(j,t)$-adjacent to $B$ and there is no operator from ${\mathcal G}(\sigma,d)$
which is $(j,t)$-adjacent to $A$ and $(i,j)$-adjacent to $B$.
The operator $f(C)$ is $(i',j')$-adjacent to $f(A)$ and $(s',t')$-adjacent to $f(B)$.
If $i',j',s',t'$ are mutually distinct, then the statement (1) from Lemma \ref{lemma2-tech} implies the existence of $C'\in {\mathcal G}(\sigma,d)$ which is 
$(s',t')$-adjacent to $f(A)$ and $(i',j')$-adjacent to $f(B)$.
Then $f^{-1}(C')$ is $(j,t)$-adjacent to $A$ and $(i,j)$-adjacent to $B$ which is impossible.
Therefore, 
$$\tau\bigl(\{i,j\}\bigr)\cap\tau\bigl(\{j,t\}\bigr)\ne\emptyset.$$
Applying the same arguments to $f^{-1}$ and $\tau^{-1}$ we establish that $\tau$ is an automorphism of $J(I,2)$.
\end{proof}

If $\abs{I}\ne 4$, then every automorphism of $J(I,2)$ is induced by a permutation on the set $I$.
In the case when $\abs{I}=4$, an automorphism of $J(I,2)$ is induced by a permutation on $I$ or 
it is the composition of an automorphism induced by a permutation and the automorphism sending every $2$-element subset $J\subset I$ to the complement $I\setminus J$.
Therefore, one of the following 
possibilities
is realised:
\begin{itemize}
\item there is a permutation $\delta$ on the set $I$ such that $f$ sends $(i,j)$-adjacent operators to $(\delta(i),\delta(j))$-adjacent operators;
\item $\abs{I}=4$ and there is a permutation $\delta$ on $I$ such that $f$ sends $(i,j)$-adjacent operators to $(i',j')$-adjacent operators,
where $\{i',j'\}=I\setminus\{\delta(i),\delta(j)\}$.
\end{itemize}
In the second case, $\tau$ transfers the collection of $2$-element subsets of $I$ containing a certain $i\in I$ to 
the collection of $2$-element subsets contained in a certain $3$-element subset of $I$.

\subsection{The case $\abs{I}\ne 4$}
Let $i\in I$. We say that two operators $A,B\in{\mathcal G}(\sigma,d)$ are $\overline{i}$-{\it connected}
if there is a sequence 
$$A=A_{0},A_1,\dots,A_{t}=B$$
such that for every $s\in \{1,\dots,t\}$ the operators $A_{s-1},A_s$ are $(i_s,j_s)$-adjacent and $i\not\in \{i_s,j_s\}$.

Let ${\mathcal G}(i)$ be the set of all eigenspaces of operators from ${\mathcal G}(\sigma,d)$ corresponding to $a_{i}$.
If $n_i$ is finite, then ${\mathcal G}(i)$ is the Grassmannian of $n_i$-dimensional subspaces.
In the case when $n_i$ is infinite, ${\mathcal G}(i)$ is formed by all closed subspaces of codimension
$$n^{i}=\sum_{j\in I\setminus\{i\}} n_j.$$
For every $S\in {\mathcal G}(i)$ we denote by $[S]_i$ 
the set of all operators from ${\mathcal G}(\sigma,d)$ whose eigenspaces corresponding to $a_i$ coincide with $S$.
Any two operators from $[S]_i$ are  $\overline{i}$-connected. 
Conversely, if ${\mathcal X}$ is a subset of  ${\mathcal G}(\sigma,d)$, where any two elements are $\overline{i}$-connected,
then all operators from ${\mathcal X}$ have the same eigenspace corresponding to $a_i$, i.e.\ ${\mathcal X}$ is contained in a certain $[S]_i$.
Therefore, 
$$\bigl\{[S]_i: S\in {\mathcal G}(i)\bigr\}$$
can be characterised as the family of all subsets of ${\mathcal G}(\sigma,d)$
maximal with respect to the property that any two elements are $\overline{i}$-connected.

Suppose that $\tau$ (the automorphism of $J(I,2)$ associated to $f$) is induced by a permutation $\delta$ on $I$
(this holds if $\abs{I}\ne 4$).
The automorphism $f$ sends $\overline{i}$-connected operators to $\overline{\delta(i)}$-connected operators.
Therefore, for every  $S\in {\mathcal G}(i)$ there is  $S'\in {\mathcal G}(\delta(i))$ such that 
$$f\bigl([S]_i\bigr)=[S']_{\delta(i)};$$
in other words, there is a map
$$f_i:{\mathcal G}(i)\to {\mathcal G}(\delta(i))$$
satisfying 
$$f\bigl([S]_i\bigr)=\bigl[f_{i}(S)\bigr]_{\delta(i)}$$
for every $S\in {\mathcal G}(i)$.
The map $f_i$ is bijective (since $f$ is bijective).

\begin{lemma}\label{lemma3}
The map $f_i$ is adjacency preserving in both directions.
\end{lemma}

\begin{proof}
The statement is a consequence of the following fact:  
$S,T\in {\mathcal G}(i)$ are adjacent if and only if for every $t\in I\setminus\{i\}$
there are $(i,t)$-adjacent operators $A\in [S]_i$ and $B\in [T]_i$.
\end{proof}

\begin{lemma}\label{lemma4}
The permutation $\delta$ belongs to $S(d)$.
\end{lemma}

To prove Lemma \ref{lemma4} we need to shed some light on isomorphisms between Grassmann graphs 
(see \cite[Chapter 2]{Pankov-book2} for the details).

\begin{rem}\label{rem-gr}{\rm
The Grassmann graphs formed by $m$-dimensional and $m'$-dimensional subspaces of $H$ are isomorphic 
if and only if $m=m'$ or $H$ is finite-dimensional and $\dim H=m+m'$. 
In the case when $H$ is infinite-dimensional, the Grassmann graph formed by closed subspaces of codimension $m$ is isomorphic to 
the Grassmann graph of $m$-dimensional subspaces (the orthocomplementation map) and, consequently,
it is isomorphic to the Grassmann graph of $m'$-dimensional subspaces if and only if $m=m'$.
}\end{rem}

\begin{proof}[Proof of Lemma \ref{lemma4}]
We need to show that $n_{i}=n_{\delta(i)}$ for every $i\in I$.

If $n_i$ and $n_{\delta(i)}$ both are finite, then $f_i$ is an isomorphism between 
the Grassmann graphs of $n_i$-dimensional and $n_{\delta(i)}$-dimensional subspaces of $H$
(Lemma \ref{lemma3}). By Remark \ref{rem-gr}, $n_{i}=n_{\delta(i)}$ or $H$ is finite-dimensional and 
$\dim H=n_i+n_{\delta(i)}$. In the second case, we have $\abs{I}=2$ which is impossible. 

Suppose that one of $n_i,n_{\delta(i)}$, say $n_i$, is 
infinite.
If $\delta(i)\ne i$, then $n_{\delta(i)}$ is finite and $f_i$ is an isomorphism between the Grassmann graphs formed by closed subspaces of codimension $n^i$
and $n_{\delta(i)}$-dimensional subspaces of $H$ (Lemma \ref{lemma3}). 
By Remark \ref{rem-gr}, we have $n^i=n_{\delta(i)}$ which means that $\abs{I}=2$, a contradiction.
\end{proof}

By Lemma \ref{lemma4},
 we can assume that $f$ preserves each type of adjacency (otherwise, we replace $f$ by the automorphism $\delta^{-1}f$).
Then each $f_i$ is a bijective transformation of ${\mathcal G}(i)$ 
and 
$$f\biggl(\sum_{i\in I}a_iP_{X_i}\biggr)=\sum_{i\in I}a_iP_{f_i(X_i)},$$
if $X_i\in {\mathcal G}(i)$ and $H$ is the orthogonal sum of all $X_i$.
Therefore, $X\in {\mathcal G}(i)$ and $Y\in {\mathcal G}(j)$ are orthogonal 
if and only if $f_i(X)$ and $f_j(Y)$ are orthogonal.

Suppose that $\dim H=n$ is finite.
Since $\abs{I}\ge 3$, for any distinct $i,j\in I$ there is $t\in I\setminus\{i,j\}$. 
Consider the bijective transformation $h$ of ${\mathcal G}_{n-n_t}(H)$ 
which sends every $(n-n_t)$-dimensional subspace $Z$ to $f_{t}\bigl(Z^{\perp}\bigr)^{\perp}$.
For $X\in {\mathcal G}_{n_i}(H)$ and $Z\in {\mathcal G}_{n-n_t}(H)$
we have $X\subset Z$ if and only if $X$ and $Z^{\perp}$ are orthogonal. 
The latter holds if and only if $f_i(X)$ and $f_t\bigl(Z^{\perp}\bigr)$ are orthogonal or, equivalently, $f_i(X)$ is contained in $f_t\bigl(Z^{\perp}\bigr)^{\perp}$.
Therefore,
$$X\subset Z\;\Longleftrightarrow\;f_i(X)\subset h(Z).$$
By Proposition \ref{prop-Chow}, $f_i$ and $h$ are induced by the same semilinear automorphism of $H$. 
Similarly, we establish that $f_j$ and $h$ are induced by the same semilinear automorphism of $H$.
So, all $f_i$ are induced by the same semilinear automorphism $U$ of $H$.
Since $U$ sends orthogonal vectors to orthogonal vectors, it is a scalar multiple of a unitary or anti-unitary operators \cite[Proposition 4.2]{Pankov-book2}.
 For every non-zero scalar $a$ and every subspace $X\subset H$ we have  $aU(X)=U(X)$,
 i.e.\ we can assume that $U$ is unitary or anti-unitary.
This gives the claim.

Let $i$ be an index from $I$ such that  $n_i$ is infinite. 
Then $n^i=\sum_{j\in I\setminus \{i\}}n_j$ is finite.
As above,
we consider the bijective transformation $h$ of ${\mathcal G}_{n^i}(H)$ sending every $n^i$-dimensional subspace $Z$ to $f_{t}\bigl(Z^{\perp}\bigr)^{\perp}$
and show that for all $j\in I\setminus \{i\}$ the transformations $f_j$ are induced by the same semilinear automorphism $U$ of $H$.
We can assume that $U$ is unitary or anti-unitary 
(because it preserves orthogonality). 
For every $X\in {\mathcal G}(i)$ there are mutually orthogonal subspaces $X_j\in {\mathcal G}(j)$, $j\in I\setminus\{i\}$ 
whose sum is the orthogonal complement of $X$. 
Then $f_{i}(X)$ is the orthogonal complement of 
$$\sum_{j\in I\setminus\{i\}}f_{j}(X_j)=U\biggl(\sum_{j\in I\setminus\{i\}}X_j\biggr)=U\bigl(X^{\perp}\bigr)=U(X)^{\perp}$$
which implies that $f_{k}(X)=U(X)$, i.e.\ $f_k$ also is induced by $U$.

\subsection{The case $\abs{I}=4$}
Suppose that $I=\{1,2,3,4\}$ and $n_1\ge n_2\ge n_3 \ge n_4$.
Then one of the following possibilities is realised:
\begin{enumerate}[(1)]
\item $n_1\ge n_2>1$;
\item $n_1>1$ and $n_2=n_3=n_4=1$;
\item $n_1=n_2=n_3=n_4=1$.
\end{enumerate}

{\it The case (1)}. The automorphism $f$ transfers $(1,2)$-connected components to $(i,j)$-connected components for some $i,j\in I$.
Since $n_1,n_2$ both are greater than $1$, Lemma \ref{lemma1} shows that 
$$n_1=n_i, n_2=n_j\qquad\text{or}\qquad n_1=n_j, n_2=n_i.$$ 
Without loss of generality we assume that the first possibility is realised.
Then the permutation $(1,i)(2,j)$ belongs to $S(d)$ and the automorphism $(1,i)(2,j)f$ preserves the $(1,2)$-adjacency. 
We can replace $f$ by $(1,i)(2,j)f$, which means that we can assume that $f$ preserves the $(1,2)$-adjacency.
There is a one-to-one correspondence between $(1,2)$-connected components and operators from 
the conjugacy class ${\mathcal G}\bigl((\sigma,d)_{-2,+1}\bigr)$. 
In our case, $(\sigma,d)_{-2,+1}=(\sigma',d')$, where
$$\sigma'=\{a_1,a_3,a_4\},\quad d'=\{n_1+n_2,n_3,n_4\}$$
and the associated set of indices is $\{1,3,4\}$.
The automorphism  $f$ induces a bijective transformation $f'$ of ${\mathcal G}(\sigma',d')$. 
Lemma \ref{lemma0} shows that $f'$ is an automorphism of the graph $\Gamma(\sigma',d')$. 
By Subsection 3.4, for every $A\in {\mathcal G}(\sigma',d')$ we have 
$$f'(A)=U\delta'(A)U^{*},$$
where $U$ is a unitary or anti-unitary operator and $\delta'\in S(d')$ is a permutation on $\{1,3,4\}$.
Since $n_1+n_2>n_3\ge n_4$, 
we get $\delta'(1)=1$.
Therefore, $\delta'$ is identity or 
the transposition $(3,4)$.
In the second case, we have $n_3=n_4$ and 
$(3,4)\in S(d)$. The automorphism 
$(3,4)f$ of $\Gamma(\sigma,d)$
induces the automorphism $(3,4)f'$ of 
$\Gamma(\sigma',d')$. The latter 
preserves each type of adjacency.
So, we can assume that $\delta'$ is identity and $f'$ preserves each type of adjacency.
This immediately  implies that  $f$ preserves the $(3,4)$-adjacency.
Now, we show that $f$ preserves the $(1,3)$-adjacency and $(2,3)$-adjacency or interchanges  them.

Suppose that $A,B\in {\mathcal G}(\sigma,d)$ are $(1,3)$-adjacent. Consider $A',B'\in {\mathcal G}(\sigma',d')$
such that $A\in {\mathcal G}(A')$ and $B\in {\mathcal G}(B')$ (see Subsection 3.2). Then $A'$ and $B'$ also are $(1,3)$-adjacent. 
The same holds for $f'(A')$ and $f'(B')$ (since $f'$ preservers each type of adjacency).
We have
$$f(A)\in {\mathcal G}\bigl(f'(A')\bigr),\qquad f(B)\in {\mathcal G}\bigl(f'(B')\bigr)$$
and $f'(A'),f'(B')$ are $(1,3)$-adjacent.
 By Remark \ref{rem-ad}, $f(A),f(B)$ are $(1,3)$-adjacent or $(2,3)$-adjacent. 
 We apply the same arguments to a pair of $(2,3)$-adjacent operators from ${\mathcal G}(\sigma,d)$ and get the claim.

So, $\tau$ (the automorphism of $J(I,2)$ associated to $f$) preserves the collection of $2$-element subsets of $I$ containing $3$.
This guarantees that $\tau$ is induced by a permutation on $I$ (see the remark at the end of Subsection 3.3). 
Hence, we can apply the arguments from Subsection 3.4.

{\it The case (3).} 
Since $n_1=n_2=n_3=n_4$, every permutation on $I$ belongs to $S(d)$.
If $f$ transfers the $(1,2)$-adjacency  to the $(i,j)$-adjacency for some $i,j\in I$,
then the automorphism $(1,i)(2,j)f$ preserves the $(1,2)$-adjacency. 
We assume that $f$ preserves the $(1,2)$-adjacency and repeat arguments used in the case (1).

{\it The case (2)}. 
Since $n_2=n_3=n_4$, it 
suffices 
to show that $f$ transfers the $(1,2)$-adjacency to the $(1,i)$-adjacency for a certain $i\in \{2,3,4\}$.
After that, we can repeat the above 
arguments.

Suppose to the contrary 
that $f$ sends the $(1,2)$-adjacency to the $(i,j)$-adjacency for some $i,j\in \{2,3,4\}$.
Without loss of generality we can assume that $(i,j)=(3,4)$.
Then, by Lemma \ref{lemma0},
$f$ induces an isomorphism $f'$ of $\Gamma\bigl((\sigma, d)_{-2,+1}\bigr)$ to $\Gamma\bigl((\sigma, d)_{-4,+3}\bigr)$.
In the present case, $(\sigma, d)_{-2,+1}=(\sigma',d')$ and $(\sigma, d)_{-4,+3}=(\sigma'',d'')$,
where
$$\sigma'=\{a_1,a_3,a_4\},\quad d'=\{n_1+1,1,1\}$$
and
$$\sigma''=\{a_1,a_2,a_3\},\quad d''=\{n_1,1,2\}.$$
Consider the inverse isomorphism $f'^{-1}$. 
It sends the $(1,3)$-adjacency of $\Gamma(\sigma'',d'')$ to the $(t,s)$-adjacency of $\Gamma(\sigma'',d'')$
for some $t,s\in \{1,3,4\}$. 
The dimensions of the eigenspaces of operators from ${\mathcal G}(\sigma'',d'')$ corresponding to $a_1$ and $a_3$ both are  greater than $1$.
Lemma \ref{lemma0} shows that
the same holds for the dimensions of the eigenspaces of operators from ${\mathcal G}(\sigma',d')$ corresponding to $a_t$ and $a_s$.
The latter is impossible, since $d'=\{n_1+1,1,1\}$; a contradiction.

\section{Final remarks}
Theorem \ref{theorem-main} was proved in \cite{PPZ} under the assumption that $n_i>1$ for all $i\in I$.
Now, we describe briefly the reasonings from \cite{PPZ} and explain why they cannot be exploited in the general case.

Some information concerning isomorphisms between Grassmann graphs can be found in Remark \ref{rem-gr},
but we need a bit more. 
Let $X$ and $X'$ be finite-dimensional subspaces of $H$ (not necessarily of the same dimension). 
Consider the Grassmann graphs $\Gamma$ and $\Gamma'$ formed by $m$-dimensional subspaces of $X$
and $m'$-dimensional subspaces of $X'$, respectively. 
In the case when $1<m<\dim X-1$,
the graphs $\Gamma$ and $\Gamma'$ are isomorphic if and only if $\dim X=\dim X'$ and $m'$ is equal to $m$ or $\dim X-m$.
However, for 
$$m\in \{1, \dim X-1\}\qquad\text{and}\qquad m'\in\{1, \dim X'-1\}$$ 
the graphs are isomorphic even if $\dim X\ne \dim X'$
(in this case, any two distinct vertices in each of these graphs are adjacent). 

Assume for simplicity that all $n_i$ are finite. 
Let $f$ be an automorphism of $\Gamma(\sigma,d)$.
By Lemma \ref{lemma1}, for any distinct $i,j\in I$ there are distinct $i',j'\in I$ such that $f$ 
provides a one-to-one correspondence between $(i,j)$-connected components and $(i',j')$-connected components.
Recall that the restriction of $\Gamma(\sigma,d)$ to an $(i,j)$-connected component
is isomorphic to the Grassmann graph formed by $n_i$-dimensional subspaces of a certain $(n_i+n_j)$-dimensional subspace of $H$.
Therefore, $f$ induces an isomorphism between this graph and 
the Grassmann graph of $n_{i'}$-dimensional subspaces of an $(n_{i'}+n_{j'})$-dimensional subspace of $H$.

If $n_i,n_j>1$, then 
$\{n_i, n_j\}=\{n_{i'},n_{j'}\}$.
This implies the existence of a permutation $\delta\in S(d)$ such that the automorphism $\delta f$ preserves the $(i,j)$-adjacency. 
The latter automorphism induces an automorphism of $\Gamma\bigl((\sigma,d)_{-i,+j}\bigr)$.
Applying this reduction recursively we obtain  an automorphism of a Grassmann graph.
This is one of the key methods used in \cite{PPZ}.

If $n_i=1$, then one of $n_{i'},n_{j'}$, say $n_{i'}$, also is equal to $1$. 
We cannot assert that $n_j=n_{j'}$ 
in this case
(see the above remark on isomorphisms of Grassmann graphs).

There is yet another reason. 
Suppose that $I=\{1,2,3\}$ and $n_1=n_2=n_3=1$.
Then every permutation on $I$ belongs to $S(d)$ and we can assume that $f$ preserves each type of adjacency 
and, consequently, induces an automorphism of $\Gamma\bigl((\sigma,d)_{-i,+j}\bigr)$ for any distinct $i,j\in I$.
Each $\Gamma\bigl((\sigma,d)_{-i,+j}\bigr)$ is isomorphic to the Grassmann graph formed by $1$-dimensional subspaces of ${\mathbb C}^3$
and every bijective transformation of the vertex set is a graph automorphism.

\end{document}